\documentclass[]{amsart}
\usepackage[margin=1.25in]{geometry}

\usepackage{mathrsfs}
\usepackage{amsfonts}
\usepackage{latexsym,epsfig}
\usepackage{amsmath, amscd, amsthm,amssymb}
\usepackage[pdftex]{color}
\usepackage{comment}
\usepackage{marginnote}
\usepackage[colorlinks,citecolor=cyan,linkcolor=magenta]{hyperref}
\usepackage[nameinlink]{cleveref}
\usepackage{enumitem}
\usepackage{enumitem}
\usepackage{tikz}
\usepackage{tikz-cd}
\usetikzlibrary{matrix,decorations.pathmorphing,arrows}

\usepackage{todonotes}

\newcommand{\R}{\mathbb{R}}

\newcommand{\wt}{\widetilde}
\newcommand{\mc}{\mathcal}

\newcommand{\ol}{\overline}

\newcommand{\del}{\partial}

\newcommand{\FF}{\mathcal{F}}

\newcommand{\orb}{\mathcal{O}}

\newcommand\tsim{\kern-.4em\sim}

\newcommand\ssm{\smallsetminus}

\renewcommand{\phi}{\varphi}
\renewcommand{\epsilon}{\varepsilon}

\newcommand{\core}{\mathrm{core}}

\newcommand{\fr}{\mathrm{fr}}
\newcommand{\frs}{\mathrm{fr}^s}
\newcommand{\fru}{\mathrm{fr}^u}
\newcommand{\mb}{\mathbf}

\newtheorem{theorem}{Theorem}[section]
\newtheorem{lemma}[theorem]{Lemma}

\newtheorem{corollary}[theorem]{Corollary}

\theoremstyle{definition}

\newtheorem*{remark*}{Remark}

\DeclareMathOperator{\intr}{int}

\title[]{Simultaneous universal circles and continuous extension}

\author{Sérgio R. Fenley, Michael P. Landry, and Samuel J. Taylor}
\date{}                                           

\begin{document}

\begin{abstract}
Fenley proved that any foliation almost transverse to a quasigeodesic pseudo-Anosov flow in a closed atoroidal 3-manifold has the continuous extension property, meaning the inclusions of leaves into the universal cover continuously extend to their ideal boundaries. This article gives an alternate proof of an upgraded version of this: the associated Cannon-Thurston map for the flow, constructed by Frankel and Fenley, organizes all of the leafwise continuous extensions.
The proof uses the fact that the boundary of the flowspace is naturally a universal circle for the foliation.
\end{abstract}
\maketitle

\section{Introduction}
\label{sec:intro}

Let $M$ be a closed hyperbolic $3$-manifold. The universal cover $\wt M$ is isometric to hyperbolic 3-space, and the ideal boundary of $\wt M$ is a 2-sphere that we denote $S^2_\infty$. 
Given a taut foliation $\mc F$ of $M$, we denote its lift to $\wt M$ by $\wt {\mc F}$. Work of Plante \cite{plante1975foliations}, Sullivan \cite{sullivan1976cycles} and Gromov \cite{Gromov}
implies that leaves of $\wt {\mc F}$ are uniformly Gromov hyperbolic.
In particular each leaf $\lambda$ of $\wt {\mc F}$ has a naturally associated hyperbolic boundary $\del \lambda$.

We say that $\mc F$ has the \emph{continuous extension property} if for each leaf $\lambda$ of $\wt {\mc F}$, the inclusion map $i_\lambda \colon \lambda\hookrightarrow \wt M$ extends continuously to a map
\[
\mb i_\lambda \colon \lambda\cup\del \lambda\to \wt M \cup S^2_\infty.
\]
Such an extension must be unique by continuity. We will denote the restriction of $\mb i_\lambda$ to $\partial \lambda$ also by $\mb i_\lambda$. 
The \emph{limit set} of $\lambda$ is the set of accumulation 
points of $\lambda$ at infinity, and when $\mc F$ has the continuous extension property the limit
set of any $\lambda$ is continuously parameterized by $\mb i_\lambda\colon \del \lambda\to S^2_\infty$.
Cannon and Thurston famously showed that the leaves of a fibration of $M$ over $S^1$, when lifted to $\wt M$, extend continuously to $S^2_\infty$ \cite{CannonThurston}. Moreover, the continuous extensions restricted to the ideal boundaries are sphere-filling curves. The question of when a general foliation has the continuous extension property is well known (see \cite[Question 10.2]{calegari2002problems}) and has been studied extensively by Fenley \cite{fenley_limitsets, fenley1999foliations,Fen09}.

In this article we focus on the continuous extension property when $\mc F$ is transverse to a quasigeodesic almost pseudo-Anosov flow $\varphi$ on $M$. 
In this setting, Fenley showed in \cite{Fen09} that the foliation $\mc F$ has the continuous extension property. The purpose of this article is to give a very short alternate proof of an upgraded version of this that uses the Cannon-Thurston map associated to the flowspace boundary of $\phi$. 

\begin{remark*}
An almost pseudo-Anosov flow is one obtained from a pseudo-Anosov flow by a mild modification called a ``dynamic blowup.'' (Often in the literature the modifier ``almost" is instead applied to the word ``transverse," as in ``$\FF$ is almost transverse to a pseudo-Anosov flow;" the meaning is the same). For a detailed treatment of dynamic blowups, see \cite[Section 3]{LMTstrongtst}. All the results we state here for pseudo-Anosov flows apply to almost pseudo-Anosov flows as well.
\end{remark*}

\subsection{Results}
Recall that the orbit space $\orb$ of $\phi$ is an open disk 
\cite{fenley2001quasigeodesic}, possessing a natural boundary $\partial \orb$ such that the induced action $\pi_1(M) \curvearrowright \orb \cup \partial \orb$ is continuous (\cite{fenley2012ideal}).

On the one hand, Landry-Minsky-Taylor show in \cite{LMT_simuc} that the action $\pi_1(M) \curvearrowright \partial \orb$ has the structure of a \emph{universal circle} of $\mc F$ in the sense of Thurston and Calegari-Dunfield \cite{CalDun_UC}. In particular, this implies that for each leaf $\lambda$ of $\wt {\mc F}$ there is a natural quotient map $\pi_\lambda \colon \partial \orb \to \partial \lambda$ that is \emph{monotone}, meaning that each point preimage is connected. 
The interiors of the non-singleton point preimages are called the \emph{gaps} of $\pi_\lambda$, and the complement of the gaps in $\partial \orb$ is called the  \emph{core} of $\pi_\lambda$. In particular, $\pi_\lambda$ canonically identifies $\partial \lambda$ with the quotient of $\mathrm{core}(\pi_\lambda)$ obtained by identifying the endpoints of each gap.

On the other hand, Frankel \cite{frankel2015quasigeodesic} and independently Fenley \cite{Fen16} proved the existence of a continuous, $\pi_1(M)$-equivariant \emph{Cannon-Thurston map} 
\[
\mb e \colon \partial \orb \to S^2_\infty.
\]
In fact, Frankel's result assumes only that $\phi$ is quasigeodesic. 

Our main theorem is that these two pictures, linking the boundary of $\orb$ to the ideal geometry of $\mc F$ and of $M$ respectively, are compatible:

\begin{theorem}\label{mainthm}
The map $\mb e \colon \partial \orb \to S^2_\infty$ induces a continuous map $\mb i_\lambda$ making the following diagram commute
\begin{center}
\begin{tikzcd}
\mathrm{core}(\pi_\lambda) \arrow{rr}{\mb e}\arrow[swap]{rd}{\pi_\lambda}& &S^2_\infty\\
 & \del \lambda\arrow[swap]{ru}{\mb i_\lambda}& 
\end{tikzcd}
\end{center}
and continuously extending the inclusion $i_\lambda\colon \lambda\hookrightarrow \wt M$, in the sense that the map defined by
\[
\mb i_\lambda(x)=
\begin{cases}
i_\lambda(x)  &\text{ if $x\in \lambda$}\\
\mb i_\lambda(x) &\text{ if $x\in \del \lambda$}
\end{cases}
\]
is continuous.
\end{theorem}

In other words, \Cref{mainthm} says that the Cannon-Thurston map $\mb e$ simultaneously and equivariantly organizes all limit sets of leaves of $\wt{\mc F}$, each of which is canonically, continuously parametrized by a quotient of $\del \orb$. This places a significant restriction on the Cannon-Thurston map associated to any quasigeodesic pseudo-Anosov flow almost transverse to $\mc F$, which may be a starting point for researchers trying to understand the collection of all such flows. 

As a demonstration of \Cref{mainthm}'s utility, we use it to easily prove that the leafwise boundary maps are uniformly finite-to-one, which was previously not known:

\begin{corollary}[Finite fibers of $\mb i_\lambda$]
\label{cor:finite}
Let $\mc F$ be a foliation transverse to a quasigeodesic almost pseudo-Anosov flow. Then there is a constant $k \ge 1$ such that for any leaf $\lambda$ of $\wt{\mc F}$, the extension $\mb i_\lambda \colon \partial \lambda \to S^2_\infty$ has $\# \left (\mb i_\lambda^{-1}(p) \right ) \le k$ for any $p \in S^2_\infty$. 
\end{corollary}

\subsection*{Acknowledgements}
The completion of this article was supported by the National Science Foundation under Grant No. DMS--2424139, while the authors were in residence at the Simons Laufer Mathematical Sciences Institute in Berkeley, California, during the Spring 2026 semester.
Additionally, Fenley was partially supported by the Simons Foundation,
Landry was supported by NSF grant DMS--2405453, and Taylor was supported by NSF grant DMS--2503113 and the Simons Foundation.

\section{The flowspace and its boundary}
Throughout this article, $\varphi$ will denote a flow on $M$ that is both quasigeodesic and almost pseudo-Anosov, although we will occasionally highlight which properties follow from which hypothesis.

Let $\wt \phi$ be the lift of $\phi$ to $\wt M$. Almost pseudo-Anosov and quasigeodesic flows are \emph{product-covered} \cite{fenley2001quasigeodesic, calegari2006universal}, meaning that $\wt\phi$ is topologically conjugate to the vertical flow on $\R^3$. Hence the quotient of $\wt M$ by the orbits of $\wt\phi$ is a plane, called the \emph{flowspace}, which we denote by $\orb$. The action of $\pi_1(M)$ on $\wt M$ by deck transformations preserves the orbits, as well as the weak stable/unstable foliations, of $\wt \phi$. Hence the action descends to an action on $\orb$ that preserves two singular foliations, denoted $\orb^{s/u}$ for stable/unstable respectively.

Given a subset $S\subset \orb$, we use the notation $\langle S\rangle$ to refer to the preimage of $S$ under the projection $\wt M\to \orb$. For example, if $p\in \orb$, then $\langle p\rangle$ is a flowline.

The flowspace has no natural metric and should not be thought of as a geometric object. However, Fenley \cite{fenley2012ideal} and Frankel \cite{frankel2013quasigeodesic} showed (for pseudo-Anosov and quasigeodesic flows, respectively) that $\orb$ can be canonically compactified to a closed disk
\[
\ol \orb=\orb\cup \del\orb,
\]
by adding a circle $\del\orb$ in such a way that the $\pi_1$-action on $\orb$ extends continuously to one on $\ol\orb$.
Each properly embedded ray in a leaf of $\orb^{s/u}$ limits on a unique point in $\del\orb$. Hence each leaf of $\orb^{s/u}$ has a well-defined circular order on its finitely many ends coming from $\del \orb$. Indeed, this circular order can be recovered directly from $\orb$ and used to \emph{define} the circle $\partial \orb$; see \cite{frankel2013quasigeodesic} for the more general
case of quasigeodesic flows.

A copy of $\R$ properly embedded in a leaf $\ell$ of $\orb^{s}$ or $\orb^u$ 
is called a \emph{stable} or \emph{unstable slice leaf}, respectively.

\section{Endpoint maps}
\label{s.3}
Since $\phi$ is quasigeodesic, each orbit of $\wt \phi$ has a well-defined forward and backward endpoint in $S^2_\infty$. Moreover, Calegari showed \cite{calegari2006universal} that these endpoints vary continuously; that is, the equivariant maps 
\[
E^\pm\colon \wt M\to S^2_\infty
\]
that take each $x$ to the forward/backward endpoints of the orbit through $x$ are continuous. 
The maps $E^\pm$ are constant on orbits of $\wt \phi$, so they induce continuous, $\pi_1(M)$-equivariant \emph{positive and negative endpoint maps}
\[
e^\pm \colon \orb \to S^2_\infty.
\]
Frankel proves in \cite{frankel2015quasigeodesic} that the endpoint maps $e^\pm$ extend continuously and equivariantly to maps 
\[
\mb e^\pm\colon \ol\orb\to S^2_\infty.
\]
Moreover, the extensions agree when restricted to $\del\orb$, giving a continuous equivariant map 
\[
\mb e \colon \partial \orb \to S^2_\infty. 
\]
This generalizes the Cannon-Thurston theorem: $\mb e$ is surjective because its image is a closed $\pi_1(M)$-invariant set and $\pi_1(M)$ acts minimally on $S^2_\infty$.

Each leaf of $\orb^{s/u}$ is contained in $(e^{\pm})^{-1}(p)$ for some $p\in S^2_\infty$. In other words, $e^+$ is constant on leaves of $\orb^s$ and $e^-$ is constant on leaves of $\orb^u$.

\section{The lens compactification and its boundary}
\label{s.4}

We now describe some machinery developed by Fenley \cite{Fen16} in the case of quasigeodesic pseudo-Anosov flows and by Frankel \cite{frankel_closedorbits} in the general setting of quasigeodesic flows.

Fixing an equivariant identification of $\wt M$ with $\orb\times \R$, we can compactify $\wt M$ by thinking of it as the interior of the closed cylinder $\ol\orb\times \ol \R$, where $\ol\R=[-\infty,+\infty]$. The \emph{lens compactification} of $\wt M$ is
\[
\mb L=\ol\orb\times \ol \R /\{p\times \ol\R\mid p\in\del\orb\}
\]
where we collapse the annulus face of the cylinder along $\ol \R$-fibers. 
There is a natural \emph{flattening map} $f\colon \mb L\to \ol\orb$ induced by the projection of $\ol\orb\times \R$ to the first factor.

The boundary of $\mb L$ is a sphere that can be thought of as two copies of $\orb$, corresponding to $\orb\times\{\pm\infty\}$ and which we call $\orb^\pm$ respectively, together with an equatorial $\del\orb$. 
We will call this the \emph{lens boundary} and denote it by $S^2_{\mb L}$, i.e. 
\[
S^2_{\mb L}=\orb^+\cup\orb^-\cup \del\orb.
\]
There is a natural $\pi_1$-equivariant map $m\colon S^2_{\mb L}\to S^2_\infty$ defined by
\[
m(x)=\begin{cases}
e^+\circ f(p)& \text{if $p\in \orb^+$}\\
e^-\circ f(p)& \text{if $p\in \orb^-$}\\
\mb e\circ f(p)& \text{if $p\in \del\orb$.}\\
\end{cases}
\]
As Fenley explains in \cite[Section 8]{Fen16}, this continuously extends the identity map $\wt M\to \wt M$, i.e. the map $h\colon \mb L\to \wt M\cup S^2_\infty$ defined by
\[
h(x)=\begin{cases}
x& \text{if $p\in \intr \mb L=\wt M$}\\
m(x)& \text{if $p\in \del \mb L=S^2_{\mb L}$}\\
\end{cases}
\]
is continuous. 
In addition this map is $\pi_1(M)$-equivariant by construction.

\section{Shadows of leaves and continuous extension}

Let $\mc F$ be a foliation transverse to $\phi$, and denote the lift of $\mc F$ to $\wt M$ by $\wt\FF$. Since $M$ is hyperbolic, $\phi$ is \emph{transitive}, meaning it has a dense orbit \cite{Mos92, barthelme2024non}. Hence for all $p\in M$ we can find a long nearly closed orbit that can be closed up to a $\mc F$-transversal through $p$. This shows that $\mc F$ is taut.

Let $\Omega_\lambda$ be the projection of $\lambda$ to $\orb$ and $\ol \Omega_\lambda$ its closure in $\ol \orb$. This open, connected subspace of $\orb$ is called the \emph{shadow} of $\lambda$. Each component of the frontier $\fr_\orb(\Omega_\lambda)$ of $\Omega_\lambda$ in $\orb$ is a {stable} or {unstable slice leaf} (see \cite[Theorem 4.1]{Fen09}). 
We denote the union of the stable frontier components by $\frs(\Omega_\lambda)$ and the union of the unstable frontier components by $\fru(\Omega_\lambda)$.
A \emph{frontier chain} of $\Omega_\lambda$ is a sequence of frontier components $\ell_1,\dots, \ell_n$ such that $\ell_i$ shares one ideal point (in $\del \orb$) with $\ell_{i+1}$ for $1\le i\le n-1$. The number $n$ is the \emph{length} of the frontier chain. A frontier chain may be comprised of both stable and unstable slice leaves. By \cite[Lemma 6.13]{LMT_simuc}, the length of all frontier chains is finite (in fact, the length is uniformly bounded above by a constant depending only on $\phi$ and $\mc F$). 

The closure of $\Omega_\lambda$ in $\ol\orb$ is a closed disk that we denote by $\ol\Omega_\lambda$ and which we can think of as a union
\[
\ol\Omega_\lambda=\Omega_\lambda\cup \del_\infty \ol\Omega_\lambda\cup \fr^s(\Omega_\lambda)\cup \fr^u(\Omega_\lambda),
\]
where $\del_\infty \ol\Omega_\lambda$ is the \emph{limit set of $\Omega_\lambda$} defined as $\ol \Omega_\lambda \cap \partial \orb$.

The frontier of $\Omega_\lambda$ and its decomposition into chains is related to the ideal geometry of $\lambda$ as follows. 

\begin{lemma} \label{lem:omega}
The homeomorphism $\Omega_\lambda \to \lambda$ defined by $p \mapsto \langle p \rangle \cap \lambda$ extends to a continuous quotient map
\[
\omega_\lambda\colon \ol\Omega_\lambda\to \lambda\cup \del \lambda
\]
such that the restriction of $\omega_\lambda$ to $\del \ol \Omega_\lambda$ is monotone, and 
the image of each frontier chain is a single point.
\end{lemma}

\begin{proof}
The stable/unstable foliations of $\wt \varphi$ intersect $\lambda$ in a pair of singular foliations that we denote by $\lambda^{s/u}$. By results of Fenley \cite[Corollary 5.5, Proposition 6.3]{Fen09}, the endpoints of rays in leaves of $\lambda^s$ are well-defined and dense in $\del \lambda$, and the same is true for $\lambda^u$. The homeomorphism $\lambda \to \Omega_\lambda$ identifies these foliations with the restrictions of $\orb^{s/u}$ to $\Omega_\lambda$, which we donote $\Omega_\lambda^{s/u}$. 
If $r$ is a ray in $\Omega_\lambda^{s/u}$
then it determines a ray $r_\lambda = \omega_\lambda(r)$ in $\lambda^{s/u}$,
defining a map from the endpoints of stable/unstable rays in $\Omega_\lambda$ to the endpoints of stable/unstable rays in $\lambda$. In \cite[Lemma 3.3 and 3.4]{LMT_simuc} it is shown that this assignment uniquely extends to a continuous, monotone map $\partial \ol \Omega_\lambda \to \partial \lambda$, which collapses each frontier chain to a single point.

What we must show is that the total map $\omega_\lambda\colon \ol\Omega_\lambda\to \lambda\cup \del \lambda$ is continuous, and for this it is enough to show continuity on $\del\ol\Omega_\lambda$. Being a quotient map will then follow from closedness of $\omega_\lambda$. The key fact we will use is that, by the description above: 
\begin{itemize}
\item if $r$ is a ray in $\Omega_\lambda^{s/u}$ limiting to a point $x$ in $\partial \ol \Omega_\lambda$, then $\omega_\lambda(r)$
is a  ray in $\lambda^{s/u}$ limiting to the point $\omega_\lambda(x) \in \partial \lambda$.
\end{itemize}

Suppose that $x_\infty \in \fr (\Omega_\lambda)$, and consider a sequence $x_i \in \Omega_\lambda$ converging to $x_\infty$. Let $l$ be the (say, stable) frontier component of $\Omega_\lambda$ that contains $x_\infty$, and
note that the sequence $\omega_\lambda(x_i)$ escapes compact subsets of $\lambda$.
Since $l$ is transverse to $\orb^u$ outside of a closed interval (i.e. away from any segments coming from a dynamic blowup), there is a closed disk $D \subset \orb$ whose boundary meets $l$ along an interval containing $x_\infty$ such that $D \ssm l \subset \Omega_\lambda$ and $\partial D \ssm l$ is an arc whose ends are rays in $\Omega_\lambda^u$ denoted $u_1, u_2$. The endpoints of $u_1$ and $u_2$ are points in $l$ that we denote $p_1, p_2$.
Because the endpoints of the $u_i$ are in $l$, which
is a component of the frontier of $\Omega_\lambda$,
we have 
$\omega_\lambda(p_1) =  \omega_\lambda(p_2)$.
It follows that $\omega_\lambda(u_1)$ and $\omega_\lambda(u_2)$ are
rays in $\lambda^u$ with same ideal point in $\del \lambda$.
Hence, $\omega_\lambda(\partial D \ssm l)$ is a bi-infinite line in $\lambda$ whose endpoints both converge to the single point $\omega_\lambda(x_\infty) = \omega_\lambda(l)$. Since $D\ssm l$ is a neighborhood of $x_\infty$ in $\ol \Omega_\lambda$, we have $x_i \in D$ for sufficiently large $i$. The closure of $\omega_\lambda(D\ssm l)$ in $\lambda \cup \partial \lambda$
meets $\partial \lambda$ in the single point $\omega_\lambda(x_\infty) = \omega_\lambda(l)$, and  this point must be the limit of $\omega_\lambda(x_i)$ since the sequence escapes compact sets.

Next, suppose that $x_\infty \in \partial_\infty \Omega_\lambda \subset \partial \orb$. Before proceeding with the proof we need a paragraph of technical setup.

 By \cite[Proposition 3.33]{fenley2012ideal}, $x_\infty \in \partial \orb$ has a neighborhood basis in $\ol\orb$ consisting of sets bounded by \emph{polygonal paths}. In more detail, if $x_\infty$ is not the endpoint of a leaf of $\orb^s$ ($\orb^u$ respectively), then it has a countable neighborhood basis consisting of the closure in $\ol\orb$ of half-spaces $D_n$, $n \in {\mathbb{N}}$, in $\orb$ bounded by leaves of $\orb^s$ ($\orb^u$ respectively). Otherwise, $x_\infty$ is the endpoint of leaves of both $\orb^s$ and $\orb^u$. Since $M$ is atoroidal, the number of such leaves with ideal point $x_\infty$ is finite and bounded by \cite[p. 34]{fenley2012ideal}. Let $k$ be such a
bound. In this case, $x_\infty$ has a neighborhood basis such that each basis element $D_n$ is bounded by a properly embedded bi-infinite line in $\orb$ consisting of $k-2$ leaf segments alternating between stable and unstable, together with stable/unstable rays at the ends (see \cite[Definition 3.8]{fenley2012ideal}). 
In either case, let $l_n\subset \orb$ be the boundary of $D_n$. 

Continuing with the proof, let $\omega_\lambda(l_n\cap \Omega_\lambda)$ 
have boundary points $a_n, b_n$; these are well defined by the continuity of $\omega_\lambda$ along closures of stable/unstable rays.
By the continuity of $\omega_\lambda$ in $\del \ol \Omega_\lambda$, we have 
\[
\lim_{n\to\infty}a_n=\lim_{n\to\infty} b_n= \omega_\lambda(x_\infty).
\]
The sequence $\omega_\lambda(l_n\cap \Omega_\lambda)$ escapes compact sets in $\lambda$. Hence, if it does not converge to $\omega_\lambda(x_\infty)$ then it limits to a nondegenerate interval $I$ in $\del \lambda$. Since $l_n\cap \Omega_\lambda$ has a bounded number of intervals or rays in leaves of $\Omega_\lambda^{s/u}$,
some of the images of these intervals or rays in $\lambda$ converge to a nondegenerate interval in $\partial \lambda$. This contradicts the fact that the ideal points of leaves of $\lambda^s$ (or $\lambda^u$)
are dense in $\del \lambda$. Hence $\omega_\lambda(l_n\cap \Omega_\lambda)$ converges to $\omega_\lambda(x_\infty)$, so $\omega_\lambda$ is continuous at $x_\infty$. 

Having shown continuity at all points in $\fr(\Omega_\lambda)$ and $\del_\infty \Omega_\lambda$, we are done.
\end{proof}

Since $\omega_\lambda$ is constant on frontier chains of $\Omega_\lambda$, its restriction to $\partial \ol \Omega_\lambda$ induces a unique monotone map 
\[
\pi_\lambda \colon \partial \orb \to \partial \lambda
\]
that is constant on arcs of $\partial \orb$ that are spanned by the frontier chains (see \cite[Construction 3.5]{LMT_simuc}). These are the monotone maps appearing in \Cref{mainthm}. The core of $\pi_\lambda$ is the complement of its gaps, where a gap is a maximal open interval on which $\pi_\lambda$ is constant.
By construction, the core of $\pi_\lambda$ is contained in $\partial_\infty \Omega_\lambda$ (in fact, it is exactly the set of nonisolated points of $\partial_\infty \Omega_\lambda$), and $\omega_\lambda$ and $\pi_\lambda$ have equal restrictions to $\mathrm{core}(\pi_\lambda)$. By \cite[Theorem 1.2]{LMT_simuc}, each gap of $\pi_\lambda$ is spanned by a single frontier chain.

\smallskip
The following is from \cite{Fen09}; it can also be observed by considering how $\lambda$ meets an unstable leaf in $\wt M$ whose projection to $\orb$ crosses a component of $\frs (\Omega_\lambda)$.

\begin{lemma} \label{l.below}
For each component $l^s$ of $\frs (\Omega_\lambda)$, $\langle l^s \rangle$ lies below $\lambda$ in $\wt M$; similarly, for each component $l^u$ of $\fru (\Omega_\lambda)$, $\langle l^u \rangle$ lies above $\lambda$.
\end{lemma}

As a consequence, if $l^s$ is a component of $\frs (\Omega_\lambda)$ the set of points above $\lambda$ does not accumulate on any point in $\langle l^s\rangle$. From this, one can show that $\wt \FF$-leaves can be continuously extended to the lens boundary in the following sense.

\begin{lemma}
\label{l.cont1}
Suppose that $\lambda$ is a leaf of $\wt \FF$. 
The map $s\colon \Omega_\lambda\to \wt M=\intr(\mb L)$ carrying $p$ to $\langle p\rangle \cap \lambda$ continuously extends to a map $\ol s\colon \ol\Omega_\lambda\to \mb L$ defined by
\[
\ol s(p)=
\begin{cases}
s(p) &\text{ if $p\in \Omega_\lambda$}\\
(p,+\infty) &\text{ if $p\in \frs (\Omega_\lambda)$}\\
(p,-\infty)& \text{ if $p\in \fru (\Omega_\lambda)$}\\
p & \text{ if $p\in \del _\infty \ol\Omega_\lambda$}
\end{cases}
\]
which is a homeomorphism onto its image.
\end{lemma}

We remark that (a) Given a point $p\in \ol\orb$, by $(p,\pm \infty)$ we mean the image of $(p,\pm \infty)$ in $\mb L$ under the collapse $\ol\orb\times \ol \R\to\mb L$, and (b) for the last case we are using the identification of $\del\orb$ with the equator of $S^2_{\mb L}$.

\begin{proof}
First, note that the function $\ol s$ is injective by construction
because for instance the map $p \to (p,\infty)$ is injective 
from $\orb$ to $\mb L$.

For continuity, 
let $(z_n)$ be a sequence of points in $\Omega_\lambda$ converging to a point $z\in \del \ol\Omega_\lambda$. If $z\in \del_\infty \Omega_\lambda$, it is clear that the sequence $s(z_n)$ converges to $z$ by the continuity of
the map $\mb e$ of \Cref{s.3}.

Otherwise, suppose that $z\in \fr^s(\Omega_\lambda)$. Note that $(s(z_n))$ can only accumulate on $z\times \ol \R$. The sequence cannot accumulate on any point in $z\times \R$, because $\lambda$ is closed in $\wt M$ so this would force $\lambda$ to intersect $\langle z\rangle$, a contradiction. On the other hand, the sequence cannot accumulate on $(z,-\infty)$; if it did, we could find a point on $\langle z\rangle$ that is accumulated upon by forward orbits from the $s(z_n)$. This is a contradiction because $\langle l^s\rangle$ lies below $\lambda$ in $\wt M$ by \Cref{l.below}.
We conclude that $\lim_{n\to\infty} s(z_n)=(z,\infty)$. The case when $z\in \fr^u(\Omega_\lambda)$ is symmetric.

The fact that $\ol s$ is a homeomorphism onto its image follows from continuity and injectivity, since $\ol\Omega_\lambda$ is compact and 
$\mb L$ is Hausdorff.
\end{proof}

The image of $\ol s$ is $\lambda$, together with a circle in $S^2_{\mb L}$ which is the image of $\del \ol \Omega_\lambda$. We denote this circle by $\del_{\mb L}\lambda$.
We now turn to the:
\begin{proof}[Proof of \Cref{mainthm}]
Fix a leaf $\lambda$ of $\wt{\mc F}$. The maps $h$, $\ol s$, and $\pi_\lambda$ fit into the diagram below:
\begin{center}
\begin{tikzcd}
\lambda\cup\del_{\mb L}\lambda \arrow[swap]{d}{\ol s^{-1}}\arrow{r}{h} & \wt M\cup S^2_\infty \\
\ol \Omega_\lambda \arrow[swap]{d}{\omega_\lambda}\arrow[densely dotted ]{ur} & \\
 \lambda\cup \del \lambda\arrow[densely dotted, bend right,swap]{uur}{\mb i _\lambda} & 
\end{tikzcd}
\end{center}

In the first line of the diagram, the function $h$ is the restriction
of $h: \mb L \to \wt M \cup S^2_\infty$ to $\ol\Omega_\lambda$.
Recall from \Cref{s.4} that $h$ is continuous.
Since $\ol s$ is a homeomorphism from $\ol \Omega_\lambda$ 
to its image $\lambda \cup \del_L \lambda$, 
this induces a continuous map from $\ol\Omega_\lambda$ to
$\wt M \cup S^2_\infty$ making the upper triangle commute, shown in the diagram as a dotted arrow.

Since $\omega_\lambda$ is a quotient map, there is then an induced continuous map from $\lambda\cup \del \lambda$ making the lower triangle commute. We define this map to be $\mb i_\lambda$. Unwinding the definitions shows the restriction of $\mb i_\lambda$ to $\lambda$ is simply the inclusion of $\lambda$ into $\wt M$, so $\mb i_\lambda$ continuously extends the inclusion map as claimed.
Indeed, in \Cref{l.cont1} when considering 
a component $l$ of $\fr^s(\Omega_\lambda)$ all flow lines in $\wt M$
contained in $\langle l\rangle$ are either forward asymptotic or eventually fellow travel, so their forward endpoints in $S^2_\infty$ are all equal. The same is true for the image of the
ideal endpoints of $l$ in $\del \orb$
by continuity of the map $h\colon \mb L \to S^2_\infty$.
This implies that $\omega_\lambda$ induces the map $\mb i_\lambda$.

Finally, since $\core(\pi_\lambda) = \core(\omega_\lambda)$, and is a subset of $\partial_\infty  \Omega_\lambda = \ol \Omega_\lambda \cap \partial \orb$. The restriction of $h$ to $\partial \orb$ is equal to $\mb e$ (after identifying $\partial \orb$ with the equator of $S^2_{\mb L}$), the theorem follows.
\end{proof}

We conclude by giving the proof of \Cref{cor:finite}.

\begin{proof}[Proof of \Cref{cor:finite}]
Using the diagram in \Cref{mainthm}, the corollary follows from the fact that the fibers of $\mb e \colon \partial \orb \to S^2_\infty$ are uniformly bounded. This is true because for each $p \in S^2_\infty$, $e^{-1}(p)$ is a collection of points that are joined by a collection of nonintersecting leaves of $\orb^{s/u}$ whose closure in $\ol \orb$ is connected \cite[Section 6]{Fen16}. 
If this collection contains $k$ leaves, then $\orb$ contains a chain of perfect fits of length at least $k$ \cite{Fen16}. (Here, one can take the definition of a ``perfect fit" as an innermost pair of leaves having the same endpoint in $\partial \orb$.)
Hence, by \cite[Theorem C]{Fen16}, there are at least $k$ closed orbits of $\varphi$ which are pairwise freely homotopic in $M$. 
However, by \cite[Main Theorem]{Fen16}, $k$ is bounded because $\phi$ is quasigeodesic. This completes the proof.
\end{proof}

\bibliographystyle{alphaurl}
\bibliography{bibliography}

\end{document}